\documentclass{amsart}
\usepackage{amsfonts}
\usepackage{esint}
\usepackage{color}
\usepackage{graphicx}

\newtheorem{thm}{Theorem}[section]
\newtheorem{lemma}[thm]{Lemma}

\theoremstyle{definition}

\newtheorem{definition}[thm]{Definition}
\newtheorem{example}[thm]{Example}
\theoremstyle{remark}
\newtheorem{remark}[thm]{Remark}
\numberwithin{equation}{section}

\numberwithin{equation}{section}
\newcommand{\R}{\mathbb R}

\newcommand{\lam}{\lambda}
\newcommand{\Lam}{\Lambda}

\def \R {\mathbb{R}}

\def \div {\mathrm{div}}

\def \dist {\mathrm{dist}}

\newcommand{\cd}{\rightharpoonup}

\newcommand{\defeq}{\mathrel{\mathop:}=}

\begin{document}
	
\title[Maximal solutions for the $\infty$-eigenvalue problem]{Maximal solutions for the $\infty$-eigenvalue problem}

\author[J.V. da Silva, J.D. Rossi and A.M. Salort]{Jo\~{a}o Vitor da Silva, Julio D. Rossi and Ariel M. Salort}

\address{Departamento de Matem\'atica, FCEyN - Universidad de Buenos Aires and
\hfill\break \indent IMAS - CONICET
\hfill\break \indent Ciudad Universitaria, Pabell\'on I (1428) Av. Cantilo s/n. \hfill\break \indent Buenos Aires, Argentina.}

\email[J.D. Rossi]{jrossi@dm.uba.ar}
\urladdr{http://mate.dm.uba.ar/~jrossi}

\email[A.M. Salort]{asalort@dm.uba.ar}
\urladdr{http://mate.dm.uba.ar/~asalort}

\email[J. V. da Silva]{jdasilva@dm.uba.ar}



\begin{abstract}
In this article we prove that the first eigenvalue of the $\infty-$Laplacian
$$
\left\{
\begin{array}{rclcl}
  \min\{ -\Delta_\infty v,\, |\nabla v|-\lambda_{1, \infty}(\Omega)   v \}  & = & 0 & \text{in} & \Omega \\
  v & = & 0 & \text{on} & \partial \Omega,
\end{array}
\right.
$$
has a unique (up to scalar multiplication) maximal solution.
This maximal solution can be obtained as the limit as $\ell \nearrow 1$ of concave problems of the form
$$
\left\{
\begin{array}{rclcl}
  \min\{ -\Delta_\infty v_{\ell},\, |\nabla v_{\ell}|-\lambda_{1, \infty}(\Omega)   v_{\ell}^{\ell} \}  & = & 0 & \text{in} & \Omega \\
  v_{\ell} & = & 0 & \text{on} & \partial \Omega.
\end{array}
\right.
$$
In this way we obtain that the maximal eigenfunction is the unique one that is the limit of the concave problems
as happens for the usual eigenvalue problem for the $p-$Laplacian for a fixed $1<p<\infty$.
\tableofcontents
\end{abstract}	
\maketitle


\section{Introduction}\label{Intro}
Let $\Omega \subset \R^n$ be a domain, i.e., a connected, bounded and open set, with smooth boundary,
$1<p< \infty$ and $\Delta_p u \defeq \div(|\nabla u|^{p-2}\nabla u)$ (the $p$-Laplace operator). It is a well-known fact in the literature (cf. \cite{Anane} and \cite{Lindq90}) that the first eigenvalue of the following $p$-homogeneous (nonlinear) eigenvalue problem
\begin{equation}\label{eq.p}
\left\{
\begin{array}{rclcl}
  -\Delta_p u & = & \lam  |u|^{p-2} u & \text{in} & \Omega \\
  u & = & 0 & \text{on} & \partial \Omega
\end{array}
\right.
\end{equation}
can be characterized variationally as the minimizer of the Rayleigh quotient
\begin{equation}\tag{{\bf \text{p-Eigenvalue}}}\label{1er.p}
   \displaystyle \lambda_{1, p}(\Omega) \defeq \inf_{u \in W^{1,p}_0 (\Omega) \setminus \{0\}}
   \frac{\displaystyle \int_{\Omega}| \nabla u|^p dx}{ \displaystyle \int_{\Omega}| u|^p dx}>0.
\end{equation}
Eigenvalue problems have received an increasing amount of attention along of last decades by many authors (being studied mainly via variational methods) due to several connections with applied sciences, such as bifurcation theory, resonance problems, fluid and quantum mechanics, etc, (cf. \cite{GP1}, \cite{Le}, \cite{Lindq90}, as well as the book \cite{KRV}).

Notice that, without loss of generality, due to the weak maximum principle (or Harnack inequality) the eigenfunction $u_{p}$ corresponding to $\lam_{1, p}(\Omega)$ can be considered to be positive in $\Omega$, as well as, due to the $p$-homogeneity of \eqref{eq.p}, normalized such that $\displaystyle \|u_{p}\|_{L^p(\Omega)}=1$ (cf. \cite{KL}). 
Recall that the minimum in \eqref{1er.p} is achieved by the unique positive solution (up to multiplicative constants) of equation \eqref{eq.p} with $\lam(\Omega) = \lam_{1, p}(\Omega)$ (cf. \cite{KL}). This fact is known as the \emph{simplicity} of the principal eigenvalue of \eqref{eq.p} (cf. \cite{Anane} and \cite{FranLamb}).

When one takes the limit as $ p\to \infty$ in the minimization problem \eqref{1er.p} obtains
\begin{equation} \tag{{\bf \text{$\infty$-Eigenvalue}}}\label{lam.inf}
  \lam_{1,\infty}(\Omega) \defeq \lim_{p\to\infty} \sqrt[p]{\lam_{1, p}(\Omega)} =\inf_{u \in W^{1, \infty}_0(\Omega)\setminus\{0\}} \frac{\|\nabla u\|_{L^{\infty}(\Omega)}}{\|u\|_{L^{\infty}(\Omega)}}.
\end{equation}
This min-max problem presents too many solutions as it was shown in the celebrated paper \cite{JLM} (see also \cite{JLM1}). Concerning the limit equation, also in \cite{JLM} it is
proved that any family of normalized eigenfunctions $\{u_p\}_{p>1}$ to \eqref{1er.p} fulfills (up to subsequence)
$$
	u_{p}(x) \to u_\infty(x) \quad \text{uniformly in }\overline{\Omega} \quad \text{as}\,\,\, p \to \infty,
$$
where $u_\infty \in W^{1,\infty}_0 (\Omega)$ satisfies $\|u_\infty\|_{L^{\infty}(\Omega)}=1$ and
the pair $(u_\infty, \lam_{1, \infty}(\Omega))$ is a nontrivial solution to
\begin{equation}\label{eq.infty.p}
\left\{
\begin{array}{rclcl}
  \min\Big\{-\Delta_\infty v_\infty, |\nabla v_\infty|-\lam_{1,\infty}(\Omega)v_\infty\Big\} & = & 0 & \text{in} & \Omega \\
  v_\infty & = & 0 & \text{on} & \partial \Omega.
\end{array}
\right.
\end{equation}
Here solutions are understood in the viscosity sense and
$$
   \displaystyle \Delta_\infty u(x) \defeq \sum_{i, j=1}^{n} \frac{\partial u}{\partial x_j}(x)\frac{\partial^2 u}{\partial x_j \partial x_i}(x) \frac{\partial u}{\partial x_i}(x)
$$
is the well-known \textit{$\infty-$Laplace operator}. For this reason, \eqref{eq.infty.p}, its (positive) solutions and $\lam_{1,\infty}(\Omega)$ are called in the literature the \textit{$\infty$-eigenvalue problem}, \textit{$\infty-$ground states} and the \textit{first $\infty$-eigenvalue}
respectively (cf. \cite{HSY}, \cite{JLM}, \cite{JLM1} and \cite{Yu}). In addition, also in \cite{JLM}, it is given a geometrical characterization for $\lam_{1,\infty}(\Omega)$, namely,
$$
	\lam_{1,\infty}(\Omega) = \frac{1}{\mathfrak{R}}
$$
where $\mathfrak{R}>0$ is the radius of the biggest ball contained inside $\Omega$. This means that the ``principal frequency'' for the $\infty$-eigenvalue problem can be detected from the geometry of the domain.
For more references concerning the first eigenvalue for the $\infty-$eigenvalue problem we refer to \cite{Champion}, \cite{Crasta}, \cite{Ju-Li-05}, \cite{KH}, \cite{Lin}, \cite{NRSanAS} and \cite{Yu}.

In contrast with the first (zero) Dirichlet $p$-Laplace eigenfunction (cf. \cite{BK1}, \cite{FranLamb} and \cite{GP1}), problem \eqref{eq.infty.p} may have  many solutions. In fact, the simplicity of $\lam_{1, \infty}(\Omega)$ has only been established for those domains in which the distance function $u(x) = \dist(x,\partial \Omega)$ is an eigenfunction, see \cite{Yu}. Such domains include the ball, the stadium (the convex hull of two balls with the same radii) and the torus as particular examples. Here we mention that from \cite{Lin} we know that there are convex domains for which the distance function is not 
an $\infty-$eigenfunction. Nevertheless, in general domains, one cannot expect a simple first eigenvalue, since in \cite{HSY} the authors show an example of a planar domain with a dumbbell shape (two balls of the same size with a small bridge connecting them) containing (at least) three different eigenfunctions (all of them normalized by $\|u_\infty\|_{L^{\infty}(\Omega)}=1$). We also highlight that such an example solves a conjecture posed by \cite{JLM} and \cite{JLM1}.

Taking into account the fact that, in general, $\lam_{1, \infty}(\Omega)$ is not simple, the main purpose of this paper is to prove that, in spite of this lack of simplicity, there exists a {\it unique} distinguished eigenfunction $\widehat{v}$ corresponding to $\lam_{1, \infty}(\Omega)$ that arises as the limit of sub-linear (concave) problems associated to the $\infty$-eigenvalue problem. This distinguished eigenfunction $\widehat{v}$ is characterized as the only one that fulfills a maximality property: it is normalized with $\|\widehat{v}\|_{L^{\infty}(\Omega)}=1$, it verifies $\widehat{v}\geq u_\infty$ for any other solution to \eqref{eq.infty.p} with $\|u_\infty\|_{L^{\infty}(\Omega)}=1$.

Now we take a small detour and introduce for $1<q<p$ the following family of eigenvalue problems
\begin{equation}\label{eq1.1}
\left\{\begin{array}{rcll}
  -\Delta_p u & = & \Lam_{p, q}(\Omega)\|u\|_{L^q(\Omega)}^{p-q}|u|^{q-2} u & \text{ in } \Omega \\
  u & = & 0 & \text{ on }\partial\Omega,
\end{array}
\right.
\end{equation}
see \cite{BF, FranLamb}. The first eigenvalue for this problem is given by the following quantity
\begin{equation} \label{vari.pq}
	\Lam_{p, q}(\Omega) \defeq  \inf\left\{ \int_\Omega |\nabla u|^p\,dx \colon u\in W^{1, p}_0(\Omega)\,\,\,\text{with}\,\,\, \|u\|_{L^q(\Omega)}=1\right\}.
\end{equation}
As before, we will consider the corresponding eigenfunction $u_{p,q}$ being positive in $\Omega$ and normalized such that $\|u_{p, q}\|_{L^q(\Omega)}=1$.

Our first result shows that the value $\lam_{1,\infty}(\Omega)$ defined in \eqref{lam.inf} can be also obtained as the limit of the eigenvalues $\Lam_{p,q}(\Omega)$ as $p,q\to \infty$.
\begin{thm}\label{teo.1}
Let $\Omega \subset \R^n$ be a bounded domain. Then,
\begin{equation}\label{Eqlimcon}
    \displaystyle \lim_{p,q\to\infty} \sqrt[p]{\Lambda_{p, q}(\Omega)} = \lam_{1, \infty}(\Omega),
\end{equation}
where $\lam_{1, \infty}(\Omega)$ is the quantity given by \eqref{lam.inf}.
\end{thm}

In the previous result, the arguments leading to $\lam_{1, \infty}(\Omega)$ do not require any additional assumption on the divergence rates of $p$ and $q$. However, if it is imposed that $\displaystyle \frac{q}{p} \approx \ell < 1$ as $p,q\to\infty$, then we can obtain more information in this limit procedure. The following result shows that eigenfunctions $u_{p, q}$ to \eqref{eq1.1} converge uniformly to a limit function $v_\ell \in W_0^{1,\infty}(\Omega)$. Moreover, $\lam_{1, \infty}(\Omega)$ is, in fact, an ``eigenvalue'' of a certain concave eigenvalue problem.
\begin{thm}\label{teo.2}
Let $q<p$ be such that 
$$
  \displaystyle \ell \defeq \lim_{p, q \to \infty} \frac{q}{p}<1.
$$
Then, for any sequence of eigenfunctions $\{u_{p,q}\}_{p,q}$ to \eqref{eq1.1} normalized such that $\|u_{p, q}\|_{L^q(\Omega)}=1$, there exists a limit $v_{\ell}$ (up to a subsequence),
$$
	 \lim_{p, q \to \infty} u_{p, q}(x) =v_{\ell}(x)    \quad \text{ uniformly in } \overline{\Omega},
$$
which is a viscosity solution to
\begin{equation}\label{eq.infty.pq}
\left\{\begin{array}{rcll}
  \min \Big\{ -\Delta_\infty v_{\ell},\, |\nabla v_{\ell}|-\lambda_{1, \infty}(\Omega)   v_{\ell}^{\ell} \Big\} & = & 0 & \text{ in } \Omega \\
  v_{\ell} & = & 0 & \text{ on }\partial\Omega,
\end{array}
\right.
\end{equation}
with $\|v_\ell\|_{L^\infty(\Omega)}=1.$
\end{thm}

After proving this result we turn our attention to the behavior of $\{v_{\ell}\}_{\ell\in (0, 1)}$ as $\ell \nearrow 1$. We show that such a family is decreasing with $\ell$, and then there exists a limit function $ \widehat{v}$ as $\ell\nearrow 1$. As we have anticipated, this limit function $ \widehat{v}$ has some interesting properties:
\begin{enumerate}
  \item[\checkmark] $ \widehat{v}$ is a normalized eigenfunction for the $\infty-$eigenvalue problem with eigenvalue $\lam_{1,\infty}(\Omega)$;
  \item[\checkmark] $ \widehat{v}$ is maximal in the sense of being greater or equal than any other normalized solution to \eqref{eq.infty.p}.
\end{enumerate}

Our last and main result reads as follows:
\begin{thm} \label{teo.3}
Let $v_{\ell}$ be an eigenfunction of \eqref{eq.infty.pq} with corresponding ``eigenvalue'' $\lam_{1, \infty}(\Omega)$ and $\|v_\ell\|_{L^\infty(\Omega)}=1$. Then, there exists a limit function $\widehat{v} \in W_0^{1,\infty}\cap C(\bar \Omega)$,
$$
	\lim_{\ell\nearrow 1} v_{\ell}(x) =  \widehat{v}(x) \quad \text{uniformly in} \,\,\, \overline{\Omega},
$$
such that $ \widehat{v}$ is an eigenfunction for the $\infty-$eigenvalue problem normalized with $\|\widehat{v}\|_{L^\infty(\Omega)}=1.$ Furthermore, $\widehat{v}$ is the maximal solution to \eqref{eq.infty.p} in the following sense:
$$
	\widehat{v} \geq u_\infty\quad \text{ for any other solution } u_\infty \text{ of } \eqref{eq.infty.p} \text{ with
	$\| u_\infty
\|_{L^\infty(\Omega)}=1.$
}
$$
\end{thm}

{\bf Conjecture:} We conjecture that the maximal solution is the unique variational eigenfunction for the $\infty-$eigenvalue problem, that is, the whole family of normalized eigenfunctions $\{u_p\}_{p>1}$ to \eqref{1er.p} converges to $\widehat{v}$,
$$
	u_{p}(x) \to \widehat{v}(x) \quad \text{uniformly in }\overline{\Omega} \quad \text{as}\,\,\, p \to \infty.
$$
Notice that this holds trivially when one has simplicity of the first eigenvalue $\lam_{1, \infty}(\Omega)$
(this happens in a ball, a stadium and other domains) but it also holds for the counterexample to simplicity presented in \cite{HSY} where the maximal solution is also the limit of the $\{u_p\}_{p>1}$ (this is due to symmetry reasons).

\begin{remark}
One could guess the existence of a result similar to Theorem \ref{teo.3} regarding minimal solutions to \eqref{eq.infty.p}. Nevertheless, in a general context, such a minimal solution could not exist as illustrates the example presented in \cite{HSY}.
\end{remark}

\begin{remark}
Problem \eqref{eq.infty.pq} also arises as limit when $p,q \to \infty$ of the concave problems of $p-$Laplacian type
\begin{equation}\label{eq.p.qq}
\left\{
\begin{array}{rclcl}
  -\Delta_p u & = & \lam  u^{q-1}  & \text{in} & \Omega \\
  u & = & 0 & \text{on} & \partial \Omega
\end{array}
\right.
\end{equation}
with $q<p$, see \cite{CP}. This problem has a unique positive solution for every $\lambda>0$, see \cite{BK1}.

In addition, it holds that the solution $u_q$ to \eqref{eq.p.qq} that verifies $\| u_q\|_{L^\infty (\Omega)} =1$ (that exists for some value $\lambda= \lambda_q$) converges as $q\nearrow p$ ($p$ fixed) to an eigenfunction of the $p-$Laplacian (normalized with $\| u\|_{L^\infty (\Omega)} =1$), notice that we have $\lambda_q\to \lambda_{1,p}(\Omega)$ as $q\nearrow p$. Therefore, our main result, Theorem \ref{teo.3}, can be regarded as an extension to this approximation of an eigenvalue problem by sub-linear problems to the case $p=\infty$.
\end{remark}

\section{Preliminaries}

Throughout this section we will introduce some definitions and auxiliary results  we will use in this paper.
The material presented here is well-known to experts but we include some details for completeness. 

First of all, we present the notion of weak solution to
\begin{equation} \label{plap.g}
 - \Delta_p u = g_u(u) \quad \text{in} \quad \Omega,
\end{equation}
where $g_u: \R \to \R$ is the continuous function defined by
$$
   g_u(s) = \Lam_{p, q}(\Omega)\|u\|_{L^q(\Omega)}^{p-q}|s|^{q-2}s.
$$
Hereafter, since we are interested in the asymptotic behavior as $p, q \to \infty $, without loss of generality we can assume that $p >q \geq \max\{2,n\}$.

\begin{definition}\label{DefWS}
  A function $u \in W^{1,p}(\Omega) \cap C(\Omega)$ is said to be a weak solution to \eqref{eq1.1} if it fulfills
  $$
  \displaystyle \int_{\Omega} |\nabla u|^{p-2}\nabla u \cdot \nabla \phi\, dx = \int_{\Omega} g_u(u)\phi \,dx, \qquad \forall \phi \in C^{\infty}_0(\Omega).
  $$
\end{definition}
Since $p$ is large, then \eqref{eq1.1} is not singular at points where the gradient vanishes. Consequently, the mapping
$$
   x \mapsto \Delta_p \phi(x) = |\nabla \phi(x)|^{p-2}\Delta \phi(x) + (p-2)|\nabla \phi(x)|^{p-4}\Delta_{\infty} \phi(x)
$$
is well-defined, as well as it is continuous for all $\phi \in C^2(\Omega)$.

Next, we introduce the notion of viscosity solution to \eqref{eq1.1}. We refer the survey \cite{CIL} for the general theory of viscosity solutions.

\begin{definition}
  An upper (resp. lower) semi-continuous function $u: \Omega \to \R$ is said to be a viscosity sub-solution (resp. super-solution) to \eqref{eq1.1} if, whenever $x_0 \in \Omega$ and $\phi \in C^2(\Omega)$ are such that $u-\phi$ has a strict local maximum (resp. minimum) at $x_0$, then
$$
    -\Delta_p \phi(x_0) \geq g_u(\phi(x_0)) \quad (\text{resp.} \,\,\,\leq g_u(\phi(x_0))).
$$
Finally, a $u \in C(\Omega)$ is said to be a viscosity solution to \eqref{eq1.1} if it is
simultaneously a viscosity sub-solution and a viscosity super-solution.
\end{definition}

\begin{definition}\label{DefVSlimeq} A non-negative function $u \in C(\Omega)$ is said to be a viscosity solution to \eqref{eq.infty.pq} if:
\begin{enumerate}
  \item whenever $x_0 \in \Omega$ and $\phi \in C^2(\Omega)$ are such that $u(x_0) = \phi(x_0)$ and
$u(x)<\phi(x)$, when $x \neq x_0$, then
$$
   -\Delta_{\infty} \phi(x_0) \leq 0 \quad \text{or} \quad |\nabla \phi(x_0)|-\lam_{1, \infty}(\Omega)\phi^{\ell}(x_0) \leq 0.
$$
  \item whenever $x_0 \in \Omega$ and $\phi \in C^2(\Omega)$ are such that $u(x_0) = \phi(x_0)$ and
$u(x)>\phi(x)$, when $x \neq x_0$, then
$$
   -\Delta_{\infty} \phi(x_0) \geq 0 \quad \text{and} \quad |\nabla \phi(x_0)|-\lam_{1, \infty}(\Omega)\phi^{\ell}(x_0) \geq 0.
$$
\end{enumerate}
\end{definition}

The following lemmas will be used below.

\begin{lemma}\label{Lemma2.4} Assume $n<p < \infty$ and let $u \in W^{1, p}_0(\Omega)$ be a weak solution to \eqref{eq1.1}. Then $u \in C^{0, \alpha}(\Omega)$, where $\alpha = 1- \frac{n}{p}$. Moreover, the following holds
\begin{enumerate}
  \item $L^{\infty}$-bounds
  $$
  \|u\|_{L^{\infty}(\Omega)} \leq \mathfrak{C}_1,
  $$
  \item H\"{o}lder estimate
  $$
  \frac{|u(x)-u(y)|}{|x-y|^{\alpha}} \leq \mathfrak{C}_2,
  $$
 \end{enumerate}
  where $\mathfrak{C}_1$ and $\mathfrak{C}_2$ are constants depending on $n$, $\sqrt[p]{\Lambda_{p,q}(\Omega)}$ and $\|u\|_{L^q(\Omega)}$.
\end{lemma}

\begin{proof}
By multiplying \eqref{eq1.1} by $u$ and integrating by parts we obtain
$$
  \int_{\Omega} |\nabla u|^p \, dx  =   \Lambda_{p, q}(\Omega)\|u\|^p_{L^q(\Omega)} \, dx.
$$
Next, by Morrey's estimates and the previous sentence, there exists a positive constant $\mathfrak{C}=\mathfrak{C}(n,\Omega)$ independent on $p$ such that
$$
   \|u\|_{L^{\infty}(\Omega)} \leq \mathfrak{C} \|\nabla u\|_{L^p(\Omega)} \leq \mathfrak{C}
   \sqrt[p]{\Lambda_{p, q}(\Omega)}\|u\|_{L^q(\Omega)},
$$
which proves the first statement.

On the other hand, since $p > n$, combining the H\"{o}lder's inequality and Morrey's estimates we have
$$
   \frac{|u(x)-u(y)|}{|x-y|^{\alpha}} \leq \mathfrak{C} \|\nabla u\|_{L^n(\Omega)} \leq \mathfrak{C}|\Omega|^{\frac{p-n}{pn}}\|\nabla u\|_{L^p(\Omega)}\leq \hat{\mathfrak{C}}|\Omega|^{\frac{p-n}{pn}}\sqrt[p]{\Lambda_{p, q}(\Omega)}\|u\|_{L^q(\Omega)},
$$
where $\hat{\mathfrak{C}}$ depends only on $n$ and $\Omega$.
\end{proof}

The last result gives that any family of weak solutions to \eqref{eq1.1} is pre-compact. Therefore, the existence of a uniform limit for our main theorem is guaranteed.

\begin{lemma} Let $\{u_p\}_{p>1}$ be a sequence of weak solutions to \eqref{eq1.1}. Suppose that
$\sqrt[p]{\Lam_{p, q}(\Omega)}, \|u_{p,q}\|_{L^q(\Omega)} \leq \mathfrak{C}$ for all $1<p< \infty$ . Then, there exists a subsequence $p_i , q_i\to \infty$ and a limit function $u_{\infty}$ such that
$$
   \displaystyle \lim_{p_i,q_i \to \infty} u_{p_i,q_i}(x) = u_{\infty}(x)
$$
uniformly in $\Omega$. Moreover, $u_{\infty}$ is Lipschitz continuous with
$$
   \frac{|u_{\infty}(x)-u_{\infty}(y)|}{|x-y|} \leq \mathfrak{C}\limsup_{p_i,q_i \to \infty} \sqrt[p_i]{\Lambda_{p_i, q_i}(\Omega)}\|u_{p_i,q_i}\|_{L^{q_i}(\Omega)}.
$$
\end{lemma}
\begin{proof} Existence of $u_{\infty}$ as an uniform limit is a direct consequence of the Lemma \ref{Lemma2.4} combined with an Arzel\`{a}-Ascoli compactness criteria. Finally, the last statement holds by passing to the limit in the H\"{o}lder's estimates from Lemma \ref{Lemma2.4}.
\end{proof}

The following lemma establishes a relation between weak and viscosity sub and super-solutions to \eqref{eq1.1}. We include the details for completeness.

\begin{lemma}\label{EquivSols} A continuous weak sub-solution (resp. super-solution) $u \in W_{\text{loc}}^{1,p}(\Omega)$ to \eqref{eq1.1} is a viscosity sub-solution (resp. super-solution) to
$$
   -\left[|\nabla u|^{p-2} \Delta u + (p-2)|\nabla u(x)|^{p-4}\Delta_{\infty} u\right] = \Lam_{p, q}(\Omega) \|u\|_{L^{q}(\Omega)}^{p-q}|u|^{q-2}u \quad \text{in} \quad \Omega.
$$
\end{lemma}

\begin{proof} Let us proceed for the case of super-solutions. Fix $x_0 \in \Omega$ and $\phi \in C^2(\Omega)$ such that $\phi$ touches $u$ by bellow, i.e., $u(x_0) = \phi(x_0)$ and $u(x)> \phi(x)$ for $x \neq x_0$. Our goal is to establish that
$$
  -\left[|\nabla \phi(x_0)|^{p-2}\Delta \phi(x_0) + (p-2)|\nabla \phi(x_0)|^{p-4}\Delta_{\infty} \phi(x_0)\right] -g_u(\phi(x_0)) \geq 0.
$$
Let us suppose, for sake of contradiction, that the inequality does not hold. Then, by continuity there exists  $r>0$ small enough such that
$$
   -\left[|\nabla \phi(x)|^{p-2}\Delta \phi(x) + (p-2)|\nabla \phi(x)|^{p-4}\Delta_{\infty} \phi(x)\right] -g_u(\phi(x)) < 0,
$$
provided that $x \in B_r(x_0)$. Now, we define the function
$$
   \Psi(x) \defeq \phi(x)+ \frac{1}{100}\mathfrak{m}, \quad \text{ where } \quad \mathfrak{m} \defeq \inf_{\partial B_r(x_0)} (u(x)-\phi(x)).
$$
Notice that $\Psi$ verifies $\Psi < u$ on $\partial B_r(x_0)$, $\Psi(x_0)> u(x_0)$ and
\begin{equation}\label{EqPsi}
 -\Delta_p \Psi(x) < g_u(\phi(x)).
\end{equation}
By extending by zero outside  $B_r(x_0)$, we may use $(\Psi-u)_{+}$ as a test function in \eqref{eq1.1}. Moreover, since $u$ is a weak super-solution, we obtain
\begin{equation}\label{Eq3.4}
  \displaystyle \int_{\{\Psi>u\}} |\nabla u|^{p-2}\nabla u \cdot \nabla (\Psi-u) dx \geq  \int_{\{\Psi>u\}} g_u(u)(\Psi-u) dx.
\end{equation}
On the other hand, multiplying \eqref{EqPsi} by $\Psi- u$ and integrating by parts we get
\begin{equation}\label{Eq3.5}
  \displaystyle \int_{\{\Psi>u\}} |\nabla \Psi|^{p-2}\nabla \Psi \cdot \nabla (\Psi-u) dx <  \int_{\{\psi>u\}} g_u(\phi)(\Psi-u) dx.
\end{equation}
Next, subtracting \eqref{Eq3.5} from  \eqref{Eq3.4} we obtain
\begin{align} \label{exxx}
  \int\limits_{\{\Psi>u\}} (|\nabla \Psi|^{p-2}\nabla \Psi - |\nabla u|^{p-2}\nabla u) \cdot \nabla (\Psi-u) dx <  \int\limits_{\{\psi>u\}} \mathcal{G}(\phi, u)(\Psi-u)dx,
\end{align}
where we have denoted $\mathcal{G}(\phi, u)=g_u(\phi)-g_u(u)$.
Finally, since the left hand side in \eqref{exxx} is bounded by below by
$$
  \mathfrak{C}(p)\int_{\{\Psi>u\}} |\nabla \Psi- \nabla u|^pdx,
$$
and the right hand side in \eqref{exxx} is negative, we can conclude that $\Psi \leq u$ in $B_r(x_0)$. However, this contradicts the fact that $\Psi(x_0)>u(x_0)$. Such a contradiction proves that $u$ is a viscosity super-solution.

Analogously we can prove that a continuous weak sub-solution is a viscosity sub-solution.
\end{proof}

The next comparison result plays an essential role in our approach.

\begin{thm}[{\cite[Theorem 10]{CP}}]\label{CompPrin}
Let $v$ and $u$ be respectively a super-solution and a sub-solution to
\begin{equation} \label{ec.comp}
\min \Big\{ -\Delta_\infty 	w, \,|\nabla w|-\lam_{1,\infty}(\Omega) w^\ell \Big\}= 0 \quad \text{ in } \Omega
\end{equation}
Suppose that both $u$ and $v$ are strictly positive in $\Omega$, continuous up to the boundary and satisfy $u\leq v$ on $\partial \Omega$. Then $u\leq v$ in $\overline{\Omega}$.
\end{thm}

\section{Proofs of the main results}

We prove Theorem \ref{teo.1} following the ideas in \cite{JLM}.

\begin{proof}[{\bf Proof of Theorem \ref{teo.1}}]

Fix $x\in\Omega$ and consider $\delta(x)=\dist(x,\partial \Omega)$ the distance function. Recall that such a function always is a solution to the minimization problem
\begin{equation} \label{dista}
	\lam_{1, \infty}(\Omega)= \frac{\displaystyle\|\nabla \delta\|_{L^{\infty}(\Omega)}}{
	\displaystyle\|\delta\|_{L^{\infty}(\Omega)}}.
\end{equation}
However, it is not (always) a genuine eigenfunction corresponding to $\lam_{1, \infty}(\Omega)$, because, in some cases, it is not a solution to the equation \eqref{eq.infty.p} as mentioned in the Introduction.

Since $\delta(x)$ is Lipschitz continuous and satisfies $|\nabla \delta(x)|=1$ a.e. $x\in\Omega$, putting it as a test function in \eqref{vari.pq} we obtain that
$$
	 \sqrt[p]{\Lam_{p,q}(\Omega)}  = \inf_{u\in W^{1,p}_0(\Omega) \setminus \{0\}} \frac{\displaystyle\left(\int_\Omega |\nabla u|^p\,dx\right)^\frac1p}{ \displaystyle\left( \int_\Omega |u|^{q}\,dx \right)^\frac{1}{q}} \leq \frac{1}{\displaystyle\left( \int_\Omega |\delta(x)|^{q}\right)^\frac{1}{q}},
$$
which from \eqref{dista} implies that
$$
  \limsup_{p,q\to\infty} \sqrt[p]{\Lam_{p,q}(\Omega)} \leq \lam_{1, \infty}(\Omega).
$$

Now, we can consider the eigenfunction $u_{p,q}$ corresponding to $\Lam_{p,q}(\Omega)$ normalized such that $\|u_{p,q}\|_{L^q(\Omega)}=1$. Consequently,
$$
	\left(\int_\Omega |\nabla u_{p,q}|^p\,dx\right)^{\frac{1}{p}} \leq \sqrt[p]{\Lam_{p,q}(\Omega)}.
$$
Hence we have an uniform bound in $p$ and $q$. Next, fix $m>n$ and for $p>q>m$ by H\"older's inequality, we obtain
$$
	\left(\int_\Omega |\nabla u_{p,q}|^m\,dx \right)^\frac1m \leq |\Omega |^{\frac1m - \frac1p}\Lam_{p,q}(\Omega).
$$
Thus, $\{u_{p,q}\}_{p,q\geq m}$ is uniformly bounded in $W^{1, m}_0(\Omega)$, for which, up to subsequences,
 we have that
$$
  u_{p, q}\cd u_\infty \quad \text{in} \quad  W^{1,m}(\Omega)
$$
and
$$
  u_{p, q} \to  u_{\infty} \quad \text{uniformly in} \quad  C^{0, \alpha}(\Omega) \,\,\, \text{for} \,\,\,\alpha \defeq 1-\frac{n}{m}.
$$
Now, for $M, N>m$ large enough, using the weak lower semi-continuity of the $L^M$ norm and uniform convergence we get that
$$
   \frac{\|\nabla u_\infty\|_{L^M(\Omega)}}{\|u_\infty\|_{L^N(\Omega)}} \leq  \liminf_{p, q\to\infty}\frac{\displaystyle\left( \int_\Omega |\nabla u_{p, q}|^M\,dx\right)^\frac{1}{M}}{\displaystyle\left( \int_\Omega |u_{p, q}|^N \,dx\right)^{\frac{1}{N}}}.
$$
Next, multiplying and dividing by $(\int_\Omega 	|u_{p, q}|^{q}\,dx)^\frac{1}{q}$ and using H\"older's inequality we obtain that
$$
\begin{array}{rcl}
  \frac{\displaystyle\|\nabla u_\infty\|_{L^M(\Omega)}}{\displaystyle\|u_\infty\|_{L^N(\Omega)}} & \leq & \displaystyle \liminf_{p, q\to\infty}\left( |\Omega|^{\frac{1}{M}-\frac{1}{p}}\sqrt[p]{\Lam_{p,q}(\Omega)}\frac{ \|u_{p, q}\|_{L^q(\Omega)}}{  \|u_{p, q}\|_{L^N(\Omega)}}\right) \\
   & \leq & \displaystyle |\Omega|^{\frac{1}{M}}\frac{\|u_{\infty}\|_{L^{\infty}(\Omega)}}{  \|u_{\infty}\|_{L^N(\Omega)}}\left( \liminf_{p, q\to\infty} \sqrt[p]{\Lam_{p,q}(\Omega)} \right)
\end{array}
$$
for fixed values of $M,N$. Finally, letting $M,N\to\infty$ and using the variational characterization of $\lam_{1, \infty}(\Omega)$ we obtain that
$$
   \lam_{1,\infty}(\Omega) \leq \liminf_{p, q\to\infty} \sqrt[p]{\Lam_{p,q}(\Omega)}.
$$
This ends the proof.
\end{proof}

Next, we will deduce the limit equation coming from \eqref{eq1.1} as $p,q\to \infty$, provided that $\displaystyle \ell \defeq \lim_{p,q\to \infty} \frac{p}{q}$.

\begin{proof}[{\bf Proof of Theorem \ref{teo.2}}]
First, we will show that $v_{\ell}$ is a viscosity sub-solution to \eqref{eq.infty.pq}. To this end, fix $x_0 \in \Omega$ and a test function $\phi \in C^2(\Omega)$ such that $v_{\ell}(x_0) = \phi(x_0)$ and the inequality $v_{\ell}(x) < \phi(x)$ holds for $x \neq x_0$.

We want to prove that
\begin{equation}\label{eq3.2}
  - \Delta_{\infty} \phi(x_0) \leq 0 \quad \text{and} \quad |\nabla \phi(x_0)|-\lambda_{1, \infty}(\Omega) \phi(x_0)^{\ell} \leq 0.
\end{equation}
Since $u_{p, q}$ converges locally uniformly to $v_{\ell}$, there exists a sequence $x_{p, q} \to x_0$ such that $u_{p, q}-\phi$ has a local maximum at $x_{p, q}$. Moreover, since $u_{p, q}$ is a weak sub-solution (resp. viscosity sub-solution according to Lemma \ref{EquivSols}) to \eqref{eq1.1}, we have that
$$
  -\frac{|\nabla \phi(x_{p, q})|^2 \Delta \phi(x_{p, q})}{p-2} - \Delta_{\infty} \phi(x_{p, q}) \leq  \frac{1}{p-2}\left( \frac{\Lam_{p, q}(\Omega)^{\frac{1}{p-4}}\phi(x_{p, q})^{\frac{q-1}{p-4}}}{|\nabla \phi(x_{p, q})|}\right)^{p-4}.
$$
Thus, $-\Delta_{\infty} \phi(x_0) \leq 0$ as $p, q \to \infty$. Finally, if
$$
   |\nabla \phi(x_0)|-\lam_{1, \infty}(\Omega)\phi(x_0)^{\ell} > 0
$$
as $p, q \to \infty$, then the right hand side of the above sentence goes to $-\infty$, which clearly yields a contradiction. Therefore \eqref{eq3.2} holds.

Now, it remains to prove that $v_{\ell}$ is a viscosity super-solution, i.e. we must show that, for each $x_0 \in \Omega $ and $\phi \in C^2(\Omega)$ such that $v_{\ell}-\phi$ achieves a strict local minimum at $x_0$, then
\begin{equation}\label{eq3.3}
   -\Delta_\infty \phi(x_0) \geq 0 \quad \text{or} \quad |\nabla \phi(x_0)|-\lam_{1, \infty}(\Omega)\phi(x_0)^{\ell} \geq 0.
\end{equation}
Again, there exists a sequence of points $x_{p, q} \to x_0$ such that $(u_{p, q}-\phi)(x_{p, q})$ is a local minimum for each $p$ and $q$. Then, as $u_{p, q}$ is a weak super-solution (consequently a viscosity super-solution according to Lemma \ref{EquivSols}), we get
$$
  -\left[|\nabla \phi(x_{p, q})|^{p-2}\Delta \phi(x_{p, q}) + (p-2)|\nabla \phi(x_{p, q})|^{p-4}\Delta_{\infty} \phi(x_{p, q})\right] \geq \Lam_{p,q}(\Omega) \phi(x_{p, q})^{q-1}.
$$
We can assume that $|\nabla \phi(x_0)|-\lambda_{1, \infty}(\Omega)\phi(x_0)^{\ell} < 0$  since otherwise \eqref{eq3.3} clearly holds. Thus, $|\nabla \phi(x_0)|>  \lambda_{1, \infty}(\Omega) \phi(x_0)^{\ell}>0$, and hence $|\nabla \phi(x_{p, q})| > 0$ for $p$ and $q$ large enough by continuity. Thus, we may divide by $(p-2)|\nabla \phi(x_{p, q})|^{p-4}$ the previous inequality to obtain the the following relation
$$
  -\frac{|\nabla \phi(x_{p, q})|^2 \Delta \phi(x_{p, q})}{p-2} - \Delta_{\infty} \phi(x_{p, q}) \geq  \frac{1}{p-2}\left( \frac{\Lam_{p, q}(\Omega)^{\frac{1}{p-4}}\phi^{\frac{q-1}{p-4}}(x_{p, q})}{|\nabla \phi(x_{p, q})|}\right)^{p-4}.
$$
The last sentence implies that $-\Delta_{\infty} \phi(x_0) \geq 0$ as $p, q \to \infty$. Hence \eqref{eq3.3} holds.
\end{proof}

\begin{example} In order to illustrate   Theorem \ref{teo.2} let us consider $\Omega = B_1(0)$ (the unit ball centered at the origin). In this context, the infinity ground state is precisely
$$
  v(x) = 1-|x|.
$$
In fact, we have that $-\Delta_{\infty} v(x) = 0$, when $x \neq 0$. Moreover, $\lam_{1, \infty}(\Omega) = 1$. Since there are no test functions $\phi$ touching $v$ from below at $x_0 = 0$, condition (2) in the definition \ref{DefVSlimeq} is automatically fulfilled. Now, if the function
$$
  \phi(x) = 1 + \mathfrak{a}\cdot x + \text{o}(|x|^2)
$$
touches $v$ from above, then we must have
$$
  1 + \mathfrak{a}\cdot x \geq 1-|x| \quad \text{as} \quad x \to 0.
$$
Hence, $|\mathfrak{a}|\leq 1$ and consequently
$$
  |\nabla \phi(0)|-\lam_{1, \infty}(\Omega)\phi^{\ell}(0) = |\mathfrak{a}|-1\leq 0,
$$
which assures that condition (1) in the definition \ref{DefVSlimeq} is satisfied.
\end{example}

Finally, the proof of Theorem \ref{teo.3} will be a direct consequence of the following two lemmas.

\begin{lemma} Let $v_{\ell}$ be the unique viscosity solution to \eqref{eq.infty.pq}. Then,
$$
  v_{\ell} \geq v_\infty \quad \text{ in } \quad \bar \Omega,
$$
for any $v_\infty$ viscosity solution to \eqref{eq.infty.p} normalized by $\| v_\infty \|_{L^\infty (\Omega)} =1$.
\end{lemma}

\begin{proof}
Since $\ell<1$ then $v_\infty\leq v_\infty^\ell$ for any $v_\infty$ normalized viscosity solution to \eqref{eq.infty.p}. Consequently, being $\lam_{1, \infty}(\Omega)>0$ we obtain (in the viscosity sense) that
$$
	\min\{ -\Delta_\infty v_\infty, |\nabla v_\infty|-\lam_{1, \infty}(\Omega) v_\infty^\ell\}\leq 0 \quad \text{ in } \Omega,
$$
i.e., $v_\infty$ is a viscosity sub-solution to \eqref{eq.infty.pq}. Recall that both $v_\infty$ and $v_{\ell}$ verify $v_{\ell}= v_\infty=0$ on $\partial\Omega$. Hence, by the Comparison Principle for sub-linear equations (Theorem \ref{CompPrin}) we obtain that $v_{\ell} \geq v_\infty$ in the whole $\overline{\Omega}$. This finishes the proof.
\end{proof}

\begin{lemma}
For each $\ell< 1$ let $v_{\ell}$ be the unique viscosity solution to \eqref{eq.infty.pq}. Then, there exists $\widehat{v} \in  W_0^{1,\infty}(\Omega)\cap C(\bar \Omega)$ such that
\begin{equation} \label{convv}
  \widehat{v} (x) = \lim_{\ell\nearrow 1} v_{\ell}(x) \quad \text{uniformly in } \overline{\Omega}.
\end{equation}
Furthermore, $\widehat{v}$ is a viscosity solution to \eqref{eq.infty.p}.
\end{lemma}

\begin{proof}
Let us see that $\{v_{\ell}\}_{\ell \in (0, 1)}$ is monotone decreasing in $\ell$.

Let $0<\ell_1<\ell_2<1$ and $v_{\ell_i}$ be a solution of \eqref{eq.infty.pq} with $\ell=\ell_i$ for $i=1,2$ normalized such that $\|v_{\ell_i}\|_{L^\infty (\Omega)}=1$. It follows that $v_{\ell_1}^{\ell_2}<  v_{\ell_1}^{\ell_1}$. Moreover, since $\lam_{1, \infty}(\Omega)>0$ it follows (in the viscosity sense) that
$$
\begin{array}{rcl}
  \min\left\{ -\Delta_\infty v_{\ell_1}, |\nabla v_{\ell_1}|- \lam_{1, \infty}(\Omega) v_{\ell_1}^{\ell_2}\right\} & \geq & \min\left\{ -\Delta_\infty v_{\ell_1}, |\nabla v_{\ell_1}|-\lam_{1, \infty}(\Omega) v_{\ell_1}^{\ell_1} \right\} \\
   & = & 0 \\
   & = & \min\left\{ -\Delta_\infty 	v_{\ell_2}, |\nabla v_{\ell_2}|-\lam_{1, \infty}(\Omega) v_{\ell_2}^{\ell_2}\right\},
\end{array}
$$
i.e., $v_{\ell_1}$ is a viscosity super-solution to \eqref{eq.infty.pq} with $\ell=\ell_2$. Since $v_{\ell_i}=0$, $i=1,2$ on $\partial\Omega$, from the Comparison Principle for sub-linear equations (Theorem \ref{CompPrin}) we obtain that $v_{\ell_1}\geq  v_{\ell_2}$ in the whole $\overline{\Omega}$.

Finally, since $\{v_{\ell}\}_{\ell \in (0, 1)} \subset W_0^{1,\infty}(\Omega)\cap C(\bar \Omega)$ is decreasing in $\ell$ and  bounded below by any $\infty-$ground state, \eqref{convv} follows from standard uniform convergence results.

Furthermore, the fact that the limit $\widehat{v}$ satisfies \eqref{eq.infty.p} in the viscosity sense follows from uniform convergence using the same steps used in the proof of Theorem \ref{teo.2} (we leave the details to the reader).
\end{proof}

\begin{remark}
Explicit solutions for the limit problem \eqref{eq.infty.pq} for a wide class of domains including the ball and the torus, among others, can be obtained as follows. Consider the ``ridge set'' of $\Omega$ defined as
$$
\begin{array}{rcl}
  \mathcal{R}_{\Omega} & = & \Big\{x \in \Omega: \dist(x, \Omega) \,\,\, \text{is not differentiable at}\,\,\,x \Big\} \\
   & = & \Big\{x \in \Omega: \exists \, x_1, x_2 \in \partial \Omega, \,\,x_1\neq x_2, \text{s.t}\,\,|x-x_1|=|x-x_2|=\dist(x, \partial \Omega)\Big\},
\end{array}
$$
as well as the set where the distance   achieves its maximum
$$
 \displaystyle \mathcal{M}_{\Omega} = \left\{x \in \Omega: \dist(x, \partial \Omega) = \max_{x \in \Omega} \dist(x, \partial \Omega)
 \right\}.
$$
Under the previous definition, we have that if $\mathcal{R}_{\Omega} = \mathcal{M}_{\Omega}$, then
$$
  v_\ell(x) = \left[\lam_{1, \infty}(\Omega)\left(\max_{x \in \Omega} \dist(x, \partial \Omega)\right)^{\ell}\right]^{\frac{1}{1-\ell}}\dist(x, \partial \Omega)
$$
is the unique positive viscosity solution to \eqref{eq.infty.pq},  see {\cite[Proposition 19]{CP}}.

Since
$$
\displaystyle \lam_{1, \infty}(\Omega) = \frac{1}{\mathfrak{R}}
= \frac{1}{\max\limits_{x \in \Omega} \dist(x, \partial \Omega)},
$$
we get that in this case all the  $v_\ell$ coincide: for any $\ell \in (0,1)$ we have that
$$
   v_\ell (x) = \left[\lam_{1, \infty}(\Omega)\left(\max_{x \in \Omega} \dist(x, \partial \Omega)\right)^{\ell}\right]^{\frac{1}{1-\ell}} \dist(x, \partial \Omega) = \frac{1}{\mathfrak{R}}\dist(x, \partial \Omega),
$$
being $\mathfrak{R}$ the radius of the biggest ball contained inside $\Omega$.

The last expression is the unique eigenfunction corresponding to the ground state for the $\infty-$eigenvalue problem. See \cite{Yu} for a proof of the simplicity of $\lam_{1, \infty}(\Omega)$ in this case.
\end{remark}

\section{Closing remarks}

We just mention that our approach is flexible enough in order to be applied for other classes of degenerate operators of $p$-laplacian type. Some interesting examples include the following:

 \begin{enumerate}
\item  \textit{Pseudo $p$-Laplacian operator}
$$
   \displaystyle -\tilde{\Delta}_p u \defeq -\sum_{i=1}^{n} \frac{\partial }{\partial x_i}\left(\left|\frac{\partial u}{\partial x_i}\right|^{p-2}\frac{\partial u}{\partial x_i}\right).
$$
The ``pseudo''-eigenvalue problem and its corresponding limit as $p \to \infty$ for such a class of operators is studied in \cite{BK2}.

\item \textit{Anisotropic $p$-Laplacian operator}
$$
   \displaystyle -\mathcal{Q}_p u \defeq -\div(\mathbb{F}^{p-1}(\nabla u)\mathbb{F}_{\xi}(\nabla u)),
$$
where $\mathbb{F}$ is an appropriate (smooth) norm of $\R^n$ and $1<p< \infty$. The necessary tools in order to study the anisotropic eigenvalue problem, as well as its limit as $p \to \infty$ can be found in \cite{BKJ}.

   \item Degenerate non-local operators of \textit{Fractional $p$-Laplacian} type
$$
   (-\Delta)_{\mathfrak{K}}^s u(x) \defeq C_{n, p, s}.\text{P.V.} \int_{\R^n} |u(x)-u(y)|^{p-2}(u(x)-u(y))\mathfrak{K}(x, y)dy,
$$
where $\mathfrak{K}: \R^n \times \R^n \to \R$ is a general singular kernel fulfilling the following properties: there exist constants $\Lambda \geq \lambda >0$ and $\mathfrak{M}, \varsigma >0$ fulfilling the following hypothesis
\begin{itemize}
  \item[\checkmark][{\bf Symmetry}] $\mathfrak{K}(x, y) = \mathfrak{K}(y, x)$ for all $x, y \in \R^n$;
  \item[\checkmark][{\bf Growth condition}] $\lambda \leq \mathfrak{K}(x, y).|x-y|^{n+ps} \leq \Lambda$ for $x, y \in \R^n$, $x \neq y$;
  \item[\checkmark][{\bf Integrability at infinity}] $0\leq \mathfrak{K}(x, y) \leq \frac{\mathfrak{M}}{|x-y|^{n+\varsigma}}$ for $x \in B_2$ and $y \in \R^n \setminus B_{\frac{1}{4}}$.
  \item[\checkmark][{\bf Translation invariance}] $\mathfrak{K}(x+z, y+z)= \mathfrak{K}(x, y)$ for all $x,y,z \in \R^n$, $x \neq y$.
  \item[\checkmark][{\bf Continuity}] The map $x \mapsto \mathfrak{K}(x, y)$ is continuous in $\R^n \setminus \{y\}$.
\end{itemize}
Clearly this previous class of operators have as prototype to  the fractional $p$-Laplacian operator provided that $\mathfrak{K}(x, y) = |x-y|^{-(n+ps)}$. The mathematical machinery in order to study the eigenvalue problem for this class of operators can be found in the following articles \cite{FerLla}, \cite{FP} and \cite{LL}.
\end{enumerate}

\subsubsection*{Acknowledgments}
This work has been partially supported by Consejo Nacional de Investigaciones Cient\'{i}ficas y T\'{e}cnicas (CONICET-Argentina). JVS would like to thank the Dept. of Math. FCEyN, Universidad de Buenos Aires for providing an excellent working environment and scientific atmosphere during his Postdoctoral program.

\end{document}